\documentclass{article}
\usepackage{geometry}
\usepackage[all]{xy}
\usepackage{amsmath,amsthm}
\usepackage{amssymb}
\usepackage{listings}
\usepackage{xcolor}
\usepackage{color}
\usepackage{graphicx}
\usepackage{verbatim}
\usepackage{epstopdf}
\newtheorem{Theorem}{\quad Theorem}[section]
\begin{document}
\centerline{\Large{\textbf{A Queueing Model for Sleep as a Vacation}}}
\begin{center}
\quad

\textbf{Nian LIU}

\quad
School of Mathematics and Statistics

Central South University

Changsha 410083, Hunan, China

\quad

\textbf{Myron HLYNKA}

\quad

Department of Mathematics and Statistics

University of Windsor

Windsor, Ontario, Canada N9B 3P4

\quad
\end{center}

\begin{abstract}
Vacation queueing systems are widely used as an extension of the classical queueing theory. We consider both working vacations and regular vacations in this paper, and compare systems with  vacations to the regular $M/M/1$ system via mean service rates and expected numbers of customers, using matrix-analytic methods. 
\end{abstract}



\quad

\section{Introduction}

In an article in the journal Science in 2013, Xie et al. (\cite{Xie}) stated ``the restorative function of sleep may be a consequence of the enhanced removal of potentially neurotoxic waste products that accumulate in the awake central nervous system'' indicating the value of sleep in changing the parameters of the brain's functioning. We can choose to consider the brain as a server in a queueing system which decreases its service rate over time but recovers after it has a rest (vacation). 

Vacation queueing systems have been studied by many authors with different models (\cite{Dos}, \cite{TianZhang}, \cite{GuoHas}, \cite{IsiO}, \cite{ZhHou}). Working vacations, introduced by Servi and Finn(2002)\cite{SerFin}, refer to a time  period, during which the service slows but does not stop. Servers would gradually get exhausted during continuous work, but the service rate could increase after a vacation of the server. We include two types of systems in this paper. The first kind of system is the regular $M/M/1$ system, in which the server works without vacations, and the service rate is a constant with a relatively low value  \cite{GroHar}. The other kind of system also has exponential interarrival and service times. However, the service rate changes after each state transition. When the service rate decreases to a certain value, the server stops working and has a vacation, after which the service rate would return to the highest level.

To compare the performances of different queueing systems, two commonly used measures  are the expected waiting time of a customer, $E(W)$), and the expected number of customers $E(L)$ in the system. These are related via Little's formula  \cite{GroHar}. In this paper, we only use $E(L)$ to measure performance of different queueing systems. Values of $E(L)$ are obtained using matrix-analytic methods (\cite{LatG}, \cite{HeQ}).

We show that the system with vacations performs better than the regular $M/M/1$ system under certain conditions.

\section{Quasi Birth and Death Processes with 4 Phases}
In this section, we compare a queueing system with working vacations with a regular $M/M/1$ system having a constant service rate.

Consider the decrease of service rate over time as working vacations (\cite{SerFin}), during which the server works with lower efficiency. The number of customers in the system and states of the server form a continuous time Markov process $\{(X(t),Y(t)),t\ge 0\}$, where $X$ is the level variable (number of customers) and $Y$ is the phase variable (server efficiency with low values indicating a high efficiency). Each state $(n,i)$ with $X(t)=n>0$ and $Y(t)=i<4$ moves to $(n-1,i+1)$ with rate $\mu_i$, or to $(n+1,i+1)$ with rate $\lambda$. For all $n\in \mathbb{N}$, state $(n,4)$ will always go to $(n+2,1)$ with rate $\lambda/2$. State $(0,i)$ will always go to  $(1,i+1)$ with rate $\lambda$, $i=1,2,3$. Set $\mu_1=\mu$, $\mu_2=a\mu$ and $\mu_3=b\mu$ ($0<b<a<1$). The process can be shown by the following network.
 The system  takes a working vacations when $Y=2$ or $3$, having a regular vacation with interval $\sim Exp(\lambda/2)$ when $Y=4$.

\[
\xymatrix{
(0,1)\ar[rd]^(0.3)\lambda & \ar[ld]_(0.3) {\mu}(1,1)\ar[rd]^(0.3)\lambda & \ar[ld]_(0.3) {\mu}(2,1)\ar[rd]^(0.3)\lambda & \ar[ld]_(0.3){\mu}(3,1) &  \dots \\
(0,2)\ar[rd]^(0.3)\lambda & \ar[ld]_(0.3){a\mu}(1,2)\ar[rd]^(0.3)\lambda & \ar[ld]_(0.3){a\mu}(2,2)\ar[rd]^(0.3)\lambda & \ar[ld]_(0.3){a\mu}(3,2) &  \dots \\
(0,3)\ar[rd]^(0.3)\lambda & \ar[ld]_(0.3){b\mu}(1,3)\ar[rd]^(0.3)\lambda & \ar[ld]_(0.3){b\mu}(2,3)\ar[rd]^(0.3)\lambda & \ar[ld]_(0.3){b\mu}(3,3) &  \dots \\
(0,4) \ar"1,3"|{\frac{\lambda}{2}} & (1,4) \ar"1,4"|{\frac{\lambda}{2}} & (2,4)\ar"1,5"|{\frac{\lambda}{2}} & (3,4) &  \dots
}
\]

The motivation for the model is that each arrival or service completion takes time and reduces the server's efficiency so Y (phase) increases at each step ($i=1,2,3$). For $i=4$, the server is exhausted and even though there may be customers to be served, the server takes a vacation long enough for 2 more customers to arrive, and then begins service with renewed vigor and first level of efficiency. Setting up the model in this way keeps the transition between states exponential at all times. The interarrival rate for customers is $\lambda$ so the expected time between until the next customer is $1/\lambda$. The expected time for two customers to arrive is $2/\lambda$ so we take the arrival rate to be $\lambda/2$ to move from state $(n,4)$ to state $(n+2,1)$ (vacation time). Another approach could have used the sum of two exponentials (each with rate $\lambda$) but we can keep our model simpler  by using rate $\lambda/2$ to have 2 customers arrive. The two approaches are not identical, though the mean times are the same, but we keep our state space more tractable using our approach. 

\begin{Theorem}
The system is stable if $\lambda<   \dfrac{\frac{\mu}{\lambda+\mu}+\frac{a\mu}{\lambda+a\mu}+\frac{b\mu}{\lambda+b\mu}+0\cdot\frac{2}{\lambda}}
{\frac{1}{\lambda+\mu}+\frac{1}{\lambda+a\mu}+\frac{1}{\lambda+b\mu}+\frac{2}{\lambda}}$
\end{Theorem}
\begin{proof}
It is sufficient to consider the situation when the level is large as that determines the stability condition. 
For states with phase variable $Y=i$ ($i=1,2,3,4$), and level $X$ large, let $v_i$ be the state transition rate, and let $w_i$ be proportion of sojourn time in those states.
\begin{equation}
w_i=\frac{\frac{1}{v_i}}{\sum_{i=1}^4\frac{1}{v_i}}
\end{equation}
where $v_1=\lambda+\mu$, $v_2=\lambda+a\mu$, $v_3=\lambda+b\mu$, $v_4=\lambda/2$.

The average service rate of the system (for large level $X$) should be calculated as a weighted average.

\begin{align}
\bar \mu &= \sum_{i=1}^4w_i\mu_i \notag\\
&= \frac{\frac{\mu}{\lambda+\mu}+\frac{a\mu}{\lambda+a\mu}+\frac{b\mu}{\lambda+b\mu}+0\cdot\frac{2}{\lambda}}
{\frac{1}{\lambda+\mu}+\frac{1}{\lambda+a\mu}+\frac{1}{\lambda+b\mu}+\frac{2}{\lambda}} 
\end{align}
The system is stable if $\lambda<\bar{\mu}$. The result follows. \end{proof}

We note in the previous proof that $\bar{\mu}$ is a function of $\lambda$. To emphasize this, we define
\begin{equation*}
g(\lambda) \stackrel{\bigtriangleup}{=}  \frac{\frac{\mu}{\lambda+\mu}+\frac{a\mu}{\lambda+a\mu}+\frac{b\mu}{\lambda+b\mu}+0\cdot\frac{2}{\lambda}}
{\frac{1}{\lambda+\mu}+\frac{1}{\lambda+a\mu}+\frac{1}{\lambda+b\mu}+\frac{2}{\lambda}}.
\end{equation*}

Unfortunately, for our 4 phase model, it turns out that regardless of $\lambda$, $\mu$, $a$, $b$, the expected number of customers will be shorter under a regular $M/M/1$ model with service rate $b\mu$ than under our model that allows for a vacation, at the cost of two customers arriving. We prove this as follows.   

\begin{Theorem}
 For the 4 phase model which is stable (i.e. $g(\lambda)>\lambda$), $g(\lambda)$ is always smaller than $b\mu$.
\end{Theorem}

\begin{proof}
First note that in our 4 phase model, states  (0,1), (0,2) and (1,1) are not recurrent. Further, the average service rate that appears for large level $X$ is an upper bound on the rate for small levels (like 1). So we will work with the service rate for large levels. Now 
\begin{equation*}
g(\lambda)-b\mu=-\mu\cdot\frac{\lambda^3(4b-a-1)+2\lambda^2\mu(ab+b-a+2b^2)+3\lambda\mu^2b^2(a+1)+2ab^2\mu^3}{3\lambda\mu^2(a+b+ab)+4\lambda^2\mu(a+b+1)+5\lambda^3+2ab\mu^3}
\end{equation*}
The denominator is always positive so we define 
\begin{equation*}
f(\lambda) \stackrel{\bigtriangleup}{=}\lambda^3(4b-a-1)+2\lambda^2\mu(ab+b-a+2b^2)+3\lambda\mu^2b^2(a+1)+2ab^2\mu^3,
\end{equation*}
Note that $g(\lambda)-b\mu<0 \Leftrightarrow f(\lambda)>0$.\\
We will view the situation graphically by considering $f(\lambda)$ which is usually a cubic in $\lambda$. \\
Case 1: $4b-a-1=0$. Then $f(\lambda)$ becomes a quadratic. Also $a=4b-1$. 
The coefficient of $\lambda^2$ in $f(\lambda)$ is $2\mu(ab+b-a+2b^2)=2\mu(a(b+1)+b+2b^2)$, which is $>0$, os the quadratic is convex. The two real roots of $f(\lambda)=0$ are $-b\mu$ and $-\frac{b\mu(4b-1)}{6b^2-4b+1}$, which are both negative. So the value of $f(\lambda)$ is positive for value of $\lambda$ which is greater than the largest root so $f(\lambda)>0$ for $\lambda>0$ , as desired. \\
Case 2: 
When $4b-a-1 < 0$, $f(\lambda)$ is a cubic with a negative coefficient for the $\lambda^3$ term.  
Let $A=\sqrt{9a^2b^2-4a^2b-14ab^2+4a^2-4ab+9b^2}$. The 3 roots of $f(\lambda)=0$ are $-b\mu$,$-\frac{\mu(3ab-2a+3b+A)}{2(4b-a-1)}$ and $-\frac{\mu(3ab-2a+3b-A)}{2(4b-a-1)}$. Two of the three roots of $f(\lambda)$ are negative with the largest root $-\frac{\mu(3ab-2a+3b+A)}{2(4b-a-1)}$.
So the cubic $f(\lambda)$ will be positive between the second largest root and the largest root, after which it becomes negative. But for $\lambda$ greater than the largest root, we have $g(\lambda)>\lambda$ so we are outside the stable region of the system. So our result is still true.\\
Case 3:
 When $4b-a-1 > 0$, $f(\lambda)$ is a cubic with a positive coefficient for the $\lambda^3$ term.  Again, we get 3 roots of $f(\lambda)=0$. The largest of the three roots is $-\frac{\mu(3ab-2a+3b-A)}{2(4b-a-1)}$.  However, the largest root would be a negative number under the following analysis.

\begin{align}
4b-a-1>0 &\Rightarrow b<a<4b-1 \notag \\
&\Rightarrow b<4b-1 \notag \\
&\Rightarrow b\in(\frac{1}{3},1) \notag \\
0<a<1 &\Rightarrow \frac{a}{a+1}\in(0,\frac{1}{2}) \notag \\
&\Rightarrow \frac{2a}{3a+3}\in(0,\frac{1}{3}) \notag \\
&\Rightarrow b>\frac{a}{a+1} \notag \\
&\Rightarrow 3ab-2a+3b>0 \notag
\end{align}

Thus, there would be
\begin{align}
&3ab-2a+3b-\sqrt{9a^2b^2-4a^2b-14ab^2+4a^2-4ab+9b^2}<0 \notag \\
&\Leftrightarrow (3ab-2a+3b)^2<9a^2b^2-4a^2b-14ab^2+4a^2-4ab+9b^2 \notag \\
&\Leftrightarrow ab(4b-a-1)<0 \notag
\end{align}

Since $f(\lambda)$ is a cubic with a positive coefficient for $\lambda^3$, then $f(\lambda)$ must be positive for all $\lambda$ larger than the largest root of $f(\lambda)=0$ so $f(\lambda)>0$ for all $\lambda>0$.

The result follows. \end{proof}

Hence, when the service rate of a regular $M/M/1$ system equals the lowest service rate in the 4 phase system, the 4 phase system will always have a lower overall average service rate than the $M/M/1$ system. This means that the M/M/1 system will have a lower expected number of customers than the 4 phase system and there is no advantage in using the 4 phase system. As a result, we move to consider a 5 phase system. 

\section{Quasi-Birth-and-Death Process with 5 States of Service}
Add one more phase standing for $c\mu$ ($0<c<b<a<1$) as service rate to the former system, and change the constant service rate in $M/M/1$ to $c\mu$. The proportion of time when the server stops working in the new system would decrease. With fixed $a$, $b$, $c$ and $\mu$, there would be a range of $\lambda$ such that the average service rate in the new system is higher than that in $M/M/1$, and the expected number of customers would be reduced when servers take some time to rest. The statement could be proved more succinctly by numerical methods rather than analytical ones.

\subsection{Matrix-Analytic Methods for Calculating the Expected Number of Customers}
For the 5-phase system, let the states be\\
$(0,1),(0,2),(0,3),(0,4),(0,5), (1,1),(1,2),(1,3),(1,4),(1,5), (2,1),(2,2),\dots$. 
The Q-matrix (infinitesimal matrix) of the system with 5 states of service is 
\[
Q1=
\begin{pmatrix}
A_{00} & A_{01} & A_{02} &        &        &        &  \\
A_{10} & A_{11} & A_{01} & A_{02} &        &        &   \\
       & A_{10} & A_{11} & A_{01} & A_{02} &        &   \\
       &        & A_{10} & A_{11} & A_{01} & A_{02} &   \\
       &        &        & \ddots & \ddots & \ddots & \ddots
\end{pmatrix}
\]
where \begin{align*}
A_{00} =\begin{pmatrix}
-\lambda &   &   &   &   \\
   & -\lambda &   &   &   \\
   &   & -\lambda &   &   \\
   &   &   & -\lambda &   \\
   &   &   &   & -\lambda/2
\end{pmatrix}_{5\times 5}
,
A_{01} =\begin{pmatrix}
  0& \lambda &   &   &   \\
   &   0& \lambda &   &   \\
   &   &  0 & \lambda &   \\
   &   &   &  0 & \lambda\\
   &   &   &   &0
\end{pmatrix}_{5\times 5}
\\
A_{02} =\begin{pmatrix}
 0 & 0 & \dots & 0 \\
 \vdots & \vdots & \ddots & \vdots \\
 0 & 0 & \dots & 0 \\
 \lambda/2 & 0 &\dots & 0
\end{pmatrix}_{5\times 5}
,
A_{10} =\begin{pmatrix}
  0 & \mu &   &   &   \\
   &  0 & a\mu &   &   \\
   &   &  0 & b\mu &   \\
   &   &   &  0 & c\mu \\
   &   &   &   &0
\end{pmatrix}_{5\times 5}
\\
A_{11}=\begin{pmatrix}
-(\lambda+\mu) &   &   &   &   \\
   & -(\lambda+a\mu) &   &   &   \\
   &   & -(\lambda+b\mu) &   &   \\
   &   &   & -(\lambda+c\mu) &   \\
   &   &   &   & -\lambda/2
\end{pmatrix}_{5\times 5}
\end{align*}

Since states $(0,1)$, $(0,2)$ and $(1,1)$ are not positive recurrent, we delete the corresponding rows and columns from Q1. 
Let 
\[ A_0= \begin{pmatrix} A_{02} & 0 \\A_{01} & A_{02} \end{pmatrix}, A_1= \begin{pmatrix} A_{11} & A_{01} \\A_{10} & A_{11} \end{pmatrix}, A_2= \begin{pmatrix} 0 & A_{10} \\ 0 & 0 \end{pmatrix}\]
After that, the Q matrix could be written as
\begin{equation}
Q= \begin{pmatrix}
B_{11}   & B_{12} &  & &  &  \\
B_{21} & A_1 & A_0 &  &  &  \\
    & A_2 & A_1 & A_0 &   &  \\
 &  & A_2 & A_1 & A_0 &   \\
 & & &\ddots&\ddots&\ddots
\end{pmatrix}
\label{Qmatrix}
\end{equation}
where
\[
B_{11}=\begin{pmatrix}
-\lambda &0  &0  &0  &0  &-\lambda  &0\\
  0 &-\lambda &0  &0  &0  &0  &-\lambda \\
  0 &0  & -\lambda/2 &0  &0  &0  &0  \\
a\mu &0  &0  &-(\lambda+a\mu) &0  &0  &0  \\
0  &b\mu &0  &0  &-(\lambda+b\mu) &0  &0  \\
 0 &0  &c\mu &0  &0  &-(\lambda+c\mu) &0  \\
0  &0  &0  &0 &0  &0  &-\lambda/2 \\
\end{pmatrix}_{7 \times 7}
\]\[
B_{12}=\begin{pmatrix}
 0 & & & & &  & & &\\
 0 &0 & & & & & & &\\
 \lambda/2 &0 &0 & & & & & &\\
 0&0 &\lambda &0 & & & &\\
 0&0 &0 &\lambda &0 & & &\\
 0&0 &0 &0 &\lambda &0 & &\\
 0&0 &0 &0 &0 &-\lambda/2 &0  &
\end{pmatrix}_{7\times 10}
,\quad
B_{21}=\begin{pmatrix}
& &\mu &0 &0 &0 \\
& & &a\mu &0 &0 \\
& & & &b\mu & 0\\
& & & & &c\mu \\
& & & & & \\
& & & & &
\end{pmatrix}_{10 \times 7}
\]

Note that $Q$ has the form of a quasi birth and death process while $Q1$ did not. \\
Let $\vec{\pi}_0 = (\pi_{(0,3)}, \pi_{(0,4)}, \pi_{(0,5)}, \pi_{(1,2)}, \pi_{(1,3)}, \pi_{(1,4)}, \pi_{(1,5)})$.

For $j\ge1$, let $\vec{\pi}_j = (\pi_{(j+1,1)}, \dots, \pi_{(j+1,5)}, \pi_{(j+2,1)}, \dots, \pi_{(j+2,5)})$. Let $\vec{\pi}=(\vec{\pi}_0, \vec{\pi}_1,\dots)$.
From $\vec{\pi} Q = \vec{0}$, we have:
\begin{align}
\vec{\pi}_0B_{11} +\vec{\pi_1}B_{21} = 0 \label{pi1} \\
\vec{\pi}_0B_{12} +\vec{\pi}_1(A_1+RA_2) = 0 \label{pi2}
\end{align}
Also, 
\begin{align}
\vec{\pi}_j = \vec{\pi}_1R^{j-1},\quad \forall j \ge 1\label{piR} \\
R^2A_2 + RA_1 + A_0 = 0 
\end{align}
where the $R$ matrix ($10 \times 10$) can be found using iteration.
\begin{align}
R(0) &= [0],  \notag \\
R(n+1) &= -\sum\limits_{k=0,k\ne 1}^{\infty}R^k(n)A_kA_1^{-1}, n \ge 0 \\
 &= -(A_0A_1^{-1}+R^2(n)A_2A_1^{-1}). \notag
\end{align}
Let $\vec{e}$ be a column vector if 1's of various lengths, as appropriate.  
Using the expression in equation (\ref{piR}), $\vec{\pi} \vec{e} =1$ implies 
\begin{equation}\vec{\pi}_0 e + \vec{\pi}_1 (I-R)^{-1}e = 1. \label{pi3}\end{equation}
Using (\ref{pi1}), (\ref{pi2}) and (\ref{pi3}), $\vec{\pi}_0$ and $\vec{\pi}_1$ can be obtained. 
From these,  limiting probabilities for all states are obtained using (\ref{piR}).
Next
\begin{align}
E(L)&=\sum\limits_{j=1}^\infty \vec{\pi}_1 R^{j-1}(j\vec{e}+(1,1,1,1,1,2,2,2,2,2)^T) \notag \\
&=\vec{\pi}_1(\sum\limits_{j=1}^\infty jR^{j-1}\vec{e} + \sum\limits_{j=1}^\infty R^{j-1} (1,1,1,1,1,2,2,2,2,2)^T) \\
&=\vec{\pi}_1((I-R)^{-2}\vec{e}+(I-R)^{-1}(1,1,1,1,1,2,2,2,2,2)^T)\notag
\end{align}

\subsection{Numerical Example of Comparing Two Systems}
Set $a = 0.99$, $b = 0.98$ and $c = 0.1$. Then the expected number in the two systems (5 phase system vs M/M/1 with lowest service rate of the 5 phase system) in terms of $\lambda$ and $\mu$ is shown in figure \ref{3D}.

\begin{figure}[htb]
\centering
\includegraphics[scale=0.5]{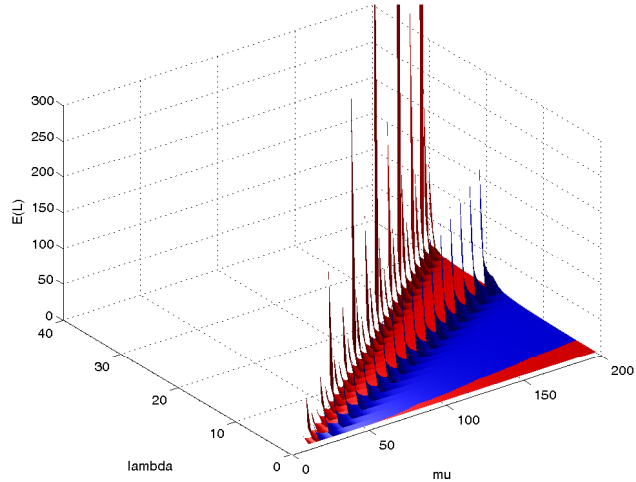}
\caption{Expected Numbers of Customers Varying with $\lambda$ \& $\mu$}
\label{3D}
\end{figure}

The expected numbers of the 5 phse system are plotted in red, and those of the $M/M/1$ system are plotted in blue. We see Figure \ref{2D} that the new system is better than the regular one only when the load $\frac{\lambda}{\mu}$ is within a certain range $(k_1, k_2)$, where $k_2 = c$. The value of $\lambda/\mu$ such that two systems have the same $E(L)$ is $k_1$.

\begin{figure}[htb]
\centering
\includegraphics[scale=0.6]{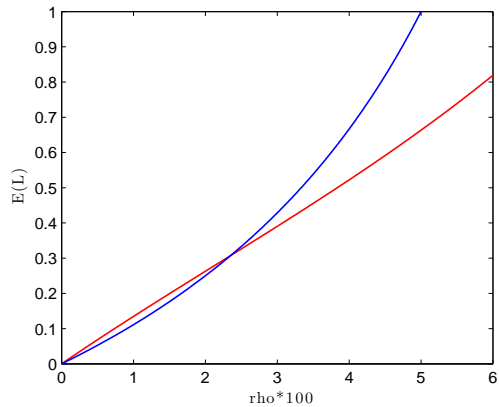}
\caption{Expected Numbers of Customers Varying with $\rho=\lambda/\mu$}
\label{2D}
\end{figure}

The value of $k_1$ could be estimated using MATLAB (see Appendix A).
For $a=0.99$, $b=0.98$ and $c=0.1$, $k_1$ is calculated to be around 0.02358. Thus, with $\lambda/\mu \in (0.02358,0.10000)$, the 5 phase system  performs better than the regular $M/M/1$ system.

\section{Conclusion}
Through comparisons on mean service rates and expected numbers of customers, we are able to state that, with two kinds of working vacations and one phase for regular rest,  a 4 phase  queueing system can never outperform the regular $M/M/1$ system with the minimal service rate. However, after we add another phase for the working vacation, it is possible for the queueing system to outperform the regular $M/M/1$ system, but only when the ratio of $\lambda$ and $\mu$ is within a certain range. The boundary of that range depends on the service rate decrease during working vacations. Basically, there is evidence that sleep is a valuable tool in allowing the brain to recuperate to its normal functioning. In a better model of the brain's recovery system, there would be a larger number of phases and the service rate would be large initially and drop off close to zero in the final phase. Our limited 5 phase model indicates that there is a real possibility for improved functioning with a good sleep cycle. The exact parameters of such a cycle would need to be estimated by a large data set, but the analysis here suggests that such a data collection is a valuable resource. 

\section{Acknowledgment}
We acknowledge funding and support  from MITACS Global Internship program, University of Windsor, Central South University, CSC Scholarship.

\quad

\newpage
\section*{Appendix}
\begin{appendix}
\section{MATLAB code serving to estimate $k_1$}
\begin{verbatim}
%mu=100
t1=zeros(1,1501);
t2=t1;
for i=0:0.01:15
k=i/100;
if k^2*(a*b+a*c+b*c+a+b+c)+2*k^3*(a+b+c+1)+3*k^4-a*b*c<0
t1(round(100*i+1))=E_cust_5ph(i,100,a,b,c);
%using floor() cause index must be a positive integer
else
t1(round(100*i+1))=NaN;
end
if k<c
t2(round(100*i+1)) = i/(c*100-i);
else
t2(round(100*i+1))=NaN;
end
end
figure;
plot(0:0.01:15,t1,'r');
axis([0 6 0 1]);
hold on
plot(0:0.01:15,t2,'b');

dif=t1-t2;
i=find(dif(1:1500).*dif(2:1501)<0);
k1=((i-1)/100+0.01*dif(i)/(dif(i)-dif(i+1)))/100;
%system with rest better than M/M/1 when k1*mu<lambda<c*m
\end{verbatim}

\section{MATLAB code of figure \ref{3D}}
\subsection{Function used for solving $E(L)$}
\begin{verbatim}
function z=E_cust_5ph(lambda,mu,a,b,c)

%find expected number of customers in a system with 4 kinds of speed
%lambda is the rate of arrival
%mu is the service rate

%e.g. lambda=1;mu=2;a=0.5;b=0.25;c=0.125

%mu=[mu1,mu2,mu3,mu4];

mu1 = mu;
mu2 = a*mu;
mu3 = b*mu;
mu4 = c*mu;
k = lambda/mu;

if k^2*(a*b+a*c+b*c+a+b+c)+2*k^3*(a+b+c+1)+3*k^4-a*b*c >= 0
    error('system is not stable, try other values of parameter')
end

a0 = [0,0,0,0,0;0,0,0,0,0;0,0,0,0,0;
    0,0,0,0,0;lambda/2,0,0,0,0];
a1 = [0,lambda,0,0,0;0,0,lambda,0,0;
    0,0,0,lambda,0;0,0,0,0,lambda;0,0,0,0,0];
a2 = [-(lambda+mu1),0,0,0,0;0,-(lambda+mu2),0,0,0;
    0,0,-(lambda+mu3),0,0;0,0,0,-lambda-mu4,0;0,0,0,0,-lambda/2];
a3 = [0,mu1,0,0,0;0,0,mu2,0,0;
    0,0,0,mu3,0;0,0,0,0,mu4;0,0,0,0,0];
A2 = [zeros(5),a3;zeros(5),zeros(5)];
A1 = [a2,a1;a3,a2];
A0 = [a0,zeros(5);a1,a0];
R = zeros(10);
for i=1:1:10^4
    T = -A0*inv(A1)-R*R*A2*inv(A1); % T is R(i), R is R(i-1)
    D = T - R;
    R = T;
    if norm(D,1)<10^(-200)
        break;
    end
end
B11 = [-lambda,0,0,0,0,lambda,0;0,-lambda,0,0,0,0,lambda;0,0,-lambda/2,0,0,0,0;
    mu2,0,0,-lambda-mu2,0,0,0;0,mu3,0,0,-lambda-mu3,0,0;0,0,mu4,0,0,-lambda-mu4,0;
    0,0,0,0,0,0,-lambda/2];
B12 = [zeros(1,10);zeros(1,10);lambda/2,zeros(1,9);
    0,0,lambda,zeros(1,7);0,0,0,lambda,zeros(1,6);0,0,0,0,lambda,zeros(1,5);
    0,0,0,0,0,lambda/2,0,0,0,0];
B21 = [zeros(5,2),a3;zeros(5,7)];
A = [B11,B12;B21,A1+R*A2];
Ac = [ones(7,1);(eye(10)-R)\ones(10,1)];
M = [A,Ac];
b = [zeros(1,17),1];
% pi*M = b  => M'*pi' = b'
pi = ((M')\(b'))';
pi1 = pi(8:17);
z = pi1*((inv(eye(10)-R)^2)*ones(10,1)+inv(eye(10)-R)*[1;1;1;1;1;2;2;2;2;2]);
\end{verbatim}
\subsection{Code for plotting}
\begin{verbatim}
a = 0.99;
b = 0.98;
c = 0.1;
z = zeros(200,200);
zz = z;
for m=1:200
    for n=1:200
        nn = n;
        k = m/nn;
        if k^2*(a*b+a*c+b*c+a+b+c)+2*k^3*(a+b+c+1)+3*k^4-a*b*c<0
           z(m,n)=E_cust_5ph(m,nn,a,b,c);
         else
            z(m,n)=NaN;
        end
        if m<c*nn
            zz(m,n) = m/(c*nn-m);
        else
            zz(m,n)=NaN;
        end
    end
end

cn=zeros(200);
for i=1:1:200
    for j=1:1:200
        cn(i,j,1)=1;%E(# with rest) is ploted in red
        cn(i,j,2)=0;cn(i,j,3)=0;
    end
end
co=zeros(200);
for i=1:1:200
    for j=1:1:200
        co(i,j,1)=0;co(i,j,2)=0;
        co(i,j,3)=1;%E(# in M/M/1) is ploted in blue
    end
end

figure;
surf(1:200,1:200,z,cn);
  axis([0 200 0 40 0 1000]);
 hold on
 surf(1:200,1:200,zz,co);
 xlabel('mu');
 ylabel('lambda');
 zlabel('expected #cust')
 shading interp;
\end{verbatim}
\end{appendix}
\end{document}